\documentclass[11pt,english,reqno]{amsart}
\usepackage[T1]{fontenc}
\usepackage[latin9]{inputenc}
\usepackage{amstext}
\usepackage{amsthm}

\makeatletter
\usepackage{amsmath,amsfonts,amssymb,amsthm,epsfig}

\usepackage{color}
\usepackage[final,allcolors=blue,colorlinks=true]{hyperref}

\def\R{\mathbb R}

\def\la{\lambda}

\def\var{\varphi}
\def\na{\nabla}
\def\Om{\Omega}  
\def\De{\Delta}      

\def\cal{\mathcal}
\def\L{\mathcal L}                                       
\def\wq{\infty}
\def\pa{\partial}

\def\rad{\text{\rm rad}}

\newcommand{\D}{{\rm d}}

\numberwithin{equation}{section}
\newtheorem{theorem}{Theorem}[section]

\newtheorem{proposition}[theorem]{Proposition}

\theoremstyle{definition}

\makeatother

\usepackage{babel}
\begin{document}
\title[on Kirchhoff equations]{Nondegeneracy of positive solutions to a Kirchhoff problem with critical Sobolev growth} 

     \author[G. Li and C.-L. Xiang]{Gongbao Li and Chang-Lin Xiang}       

\address[Gongbao Li]{School of Mathematics and Statistics and Hubei Key Laboratory of Mathematical Sciences,  Central China Normal University, Wuhan,  430079, P.R. China}
\email[]{ligb@mail.ccnu.edu.cn}
\address[Chang-Lin Xiang]{School of Information and Mathematics, Yangtze University, Jingzhou 434023, P.R. China.}
\email[]{changlin.xiang@yangtzeu.edu.cn}
\thanks{Corresponding author: Chang-Lin Xiang}
\thanks{Li is supported by  NSFC (No. 11771166), and Program for Changjiang Scholars and Innovative Research Team in University \# IRT13066. Peng is financially supported by NSFC (No. 11571130). The corresponding author Xiang is supported by NSFC (No. 11701045) and   the Yangtze Youth Fund (No. 2016cqn56).}

\begin{abstract}
In this paper, we prove uniqueness and nondegeneracy of positive solutions to the following Kirchhoff equations with critical growth \begin{eqnarray*} -\left(a+b\int_{\mathbb{R}^{3}}|\nabla u|^{2}\right)\Delta u=u^{5}, & u>0 & \text{in }\mathbb{R}^{3},\end{eqnarray*} where $a,b>0$ are positive constants. This result has potential applications in singular perturbation problems concerning Kirchhoff equaitons.
\end{abstract}

\maketitle

{\small    
\smallskip
\newline
\bigskip

\section{Introduction and main result}

In this paper, we are concerned about the nonlocal Kirchhoff type
problem
\begin{eqnarray}
-\left(a+b\int_{\R^{3}}|\na u|^{2}\right)\De u=u^{5}, & u>0 & \text{in }\R^{3},\label{eq: limiting Kirchhoff}
\end{eqnarray}
where $a,b>0$ are constants, $\De=\sum_{i=1}^{3}\pa_{x_{i}x_{i}}$
is the usual Laplacian operator in $\R^{3}$.

Denote by $D=D^{1,2}(\R^{3})$ the completion of $C_{0}^{\wq}(\R^{3})$
under the seminorm 
\[
\|\var\|_{D}^{2}\equiv\int_{\R^{3}}|\na\var|^{2}.
\]
A (weak) solution to Eq. (\ref{eq: limiting Kirchhoff}) is a function
$u\in D$ satisfying 
\[
\left(a+b\int_{\R^{3}}|\na u|^{2}\right)\int_{\R^{3}}\na u\cdot\na\var=\int_{\R^{3}}u^{5}\var
\]
for all $\var\in D$. By the Sobolev embedding $D\subset L^{6}(\R^{3})$,
all the integrals in the above equation are well defined. 

Problem (\ref{eq: limiting Kirchhoff}) and its variants have been
studied extensively in the literature. Physician Kirchhoff \cite{Kirchhoff-1883}
proposed the following time dependent wave equation 
\[
\rho\frac{\pa^{2}u}{\pa t^{2}}-\left(\frac{P_{0}}{h}+\frac{E}{2L}\int_{0}^{L}\left|\frac{\pa u}{\pa x}\right|^{2}\right)\frac{\pa^{2}u}{\pa x^{2}}=0
\]
for the first time, in order to extend the classical D'Alembert's
wave equations for free vibration of elastic strings. 
\cite{Bernstein-1940} and Pohozaev \cite{Pohozaev-1975} contributed
some early research on the study of Kirchhoff equations. Much attention
was received until the work \cite{Lions-1978} of J.L. Lions. For
more interesting results in this respect, we refer to e.g. \cite{Arosio-Panizzi-1996,DAncona-Spagnolo-1992}
and the references therein. From a mathematical point of view, the
interest of studying Kirchhoff equations comes from the nonlocality
of Kirchhoff type equations. For instance, the consideration of the
stationary analogue of Kirchhoff's wave equation leads to problem
of the type
\begin{eqnarray}
-\left(a+b\int_{\R^{3}}|\na u|^{2}\right)\De u=f(x,u) &  & \text{in }\Om\label{eq: general Kirchhoff eq.}
\end{eqnarray}
where $\Om\subset\R^{3}$ is a smooth domain. Note that the term $\left(\int|\na u|^{2}\D x\right)\De u$
depends not only on the pointwise value of $\De u$, but also on the
integral of $|\na u|^{2}$ over the whole domain. In this sense, Eqs.
(\ref{eq: limiting Kirchhoff}) and (\ref{eq: general Kirchhoff eq.})
are no longer the usual pointwise equalities. This new feature brings
new mathematical difficulties that make the study of Kirchhoff type
equations particularly interesting. We refer to e.g. \cite{Deng-Peng-Shuai-2015,Figueiredo et al-2014,Guo-2015,He-2016-JDE,He-Li-2015,He-Zou-2012,Li-Li-Shi-2012,LLPWX-2017,Li-Ye-2014,Perera-Zhang-2006,Wang-Tian-Xu-Zhang-2012}
and the references therein for mathematical researches on the existence
of solutions and many other problems. 

In this paper, we study the uniqueness and nondegeneracy of positive
solutions to problem (\ref{eq: limiting Kirchhoff}). These quantitative
properties play a fundamental role in singular perturbation problems.
Let us briefly recall some results in this respect. Kwong \cite{Kwong-1989}
established uniqueness and nondegeneracy of positive solutions to
the Schr\"odinger equations
\begin{eqnarray*}
-\De w+w=w^{q}, & w>0 & \text{in }\R^{N},
\end{eqnarray*}
see also Chang et al. \cite{Chang et al-2007}; For quasilinear Schr\"odinger
equations such as
\begin{eqnarray*}
-\Delta u-u\Delta|u|^{2}+\omega u-|u|^{q-1}u=0 &  & \text{in }\R^{N},
\end{eqnarray*}
where $\omega>0$ is a constant, $q$ is an index denoting subcritical
growth of the nonlinearity and $N\ge1$, see e.g. Selvitella \cite{Selvitella-2015},
Xiang \cite{Xiang-2016} and Adachi et al. \cite{Shinji-Masataka-Tatsuya-2016};
For fractional Schr\"odinger equations such as 
\begin{eqnarray*}
\left(-\De\right)^{s}w+w=w^{q}, & w>0 & \text{in }\R^{N},
\end{eqnarray*}
where $0<s<1\le N$ and $q$ is an index denoting subcritical growth
of the nonlinearity, see e.g. Frank and Lenzmann \cite{Frank-Lenzmann-2013}
and Frank, Lenzmann and Silvestre \cite{Frank-Lenzman-Silvestre-2016},
Fall and Valdinoci \cite{Fall-Valdinoci-2014}. For elliptic equations
with critical growth
\begin{eqnarray}
(-\De)^{s}w=w^{\frac{N+2s}{N-2s}}, & w>0 & \text{in }\R^{N},\label{eq: Yamabe problems}
\end{eqnarray}
see Ambrosetti et al. \cite{Ambrosetti-GA-Peral-1999} and D\'avila
et al. \cite{Davila-Pino-Sire-2013} for $s=1$ and $0<s<1$, respectively.
In particular, Ambrosetti and Malchiodi \cite{Ambrosetti-Malchiodi-Book}
provides a systematical research on nondegeneracy of ground states
to various types of elliptic problems together with applications in
singular perturbation problems. 

It is also known that the uniqueness and nondegeneracy of ground states
are of fundamental importance when one deals with orbital stability
or instability of ground states. It mainly removes the possibility
that directions of instability come from the kernel of the related
linear operator. The uniqueness and nondegeneracy of ground states
also play an important role in blow-up analysis for the corresponding
standing wave solutions in the corresponding time-dependent equations,
see e.g. \cite{Frank-Lenzmann-2013,Frank-Lenzman-Silvestre-2016}
and the references therein. 

For Kirchhoff problems, not much is known in this respect. Recently,
Li et al. \cite{LLPWX-2017} established uniqueness and nondegeneracy
for positive solutions to Kirchhoff equations with subcritical growth.
More precisely, they proved that the following Kirchhoff equation
\begin{eqnarray*}
-\left(a+b\int_{\R^{3}}|\na u|^{2}\right)\De u+u=u^{p}, & u>0 & \text{in }\R^{3},
\end{eqnarray*}
where $1<p<5$, has a unique nondegenerate positive radial solution.
As a counterpart to this result, we have the following theorem for
Kirchhoff equations with critical Sobolev growth.

\begin{theorem} \label{thm: uniqueness and nondegeneracy} For any
positive constants $a,b>0$, there exists a unique positive solution
$u\in D$ to equation (\ref{eq: limiting Kirchhoff}) up to scalings
and translations. Moreover, $u$ is nondegenerate in the sense that
\[
\operatorname{Ker}\L_{+}={\rm span}\left\{ u_{x_{1}},u_{x_{2}},u_{x_{3}},u/2+x\cdot\na u\right\} ,
\]
where $\L_{+}:D\to D$ is defined as 
\begin{equation}
\L_{+}\var=-\left(a+b\int_{\R^{3}}|\na u|^{2}\right)\De\var-2b\left(\int_{\R^{3}}\na u\cdot\na\var\right)\De u-5u^{4}\var\label{eq: linear operator}
\end{equation}
for all $\var\in D$.\end{theorem}

Remark that our result can also be viewed as an extension of the nondegeneracy
result of Ambrosetti \cite{Ambrosetti-GA-Peral-1999} for equations
(\ref{eq: Yamabe problems}) with $s=1$, since in the case $b=0$,
problem (\ref{eq: limiting Kirchhoff}) is reduced to (\ref{eq: Yamabe problems})
with $s=1$. 

Our notations are standard. For simplicity, we write by $\int u$
the integral $\int_{\R^{3}}u\D x$ unless otherwise stated. We also
write $u(x)=u(|x|)$ whenever $u$ is a radially symmetric function.

\section{Proof of Theorem \ref{thm: uniqueness and nondegeneracy}}

In this section we prove Theorem \ref{thm: uniqueness and nondegeneracy}.
Let 
\[
Q(x)=\frac{3^{1/4}}{(1+|x|^{2})^{1/2}}
\]
be the unique positive radial solution with $Q(0)=3^{1/4}$ that satisfies
\begin{eqnarray*}
-\De Q=Q^{5} &  & \text{in }\R^{3},
\end{eqnarray*}
 see e.g. Ambrosetti et al. \cite{Ambrosetti-GA-Peral-1999}. For
any $\la>0$ and $x_{0}\in\R^{3}$, it is direct to verify that the
functions $x\mapsto\la^{-1/2}Q((x-x_{0})/\la)$ are also solutions
to the above equation. 

\subsection{Proof of uniqueness}

\begin{proof}[Proof of uniqueness] Let $u\in D$ be an arbitrary
positive solution to Eq. (\ref{eq: limiting Kirchhoff}) and let 
\[
c=a+b\int|\na u|^{2}.
\]
Then the function $\bar{u}(x)=u(\sqrt{c}x)$ solves 
\begin{eqnarray*}
-\De\bar{u}=\bar{u}^{5} &  & \text{in }\R^{3}.
\end{eqnarray*}
Hence, the uniqueness of $Q$ implies that $\bar{u}(x)=\la^{-1/2}Q((x-x_{0})/\la)$
for some $x_{0}\in\R^{3}$ and some $\la>0$. That is, 
\[
u(x)=\la^{-1/2}Q\left(\left(\frac{x}{\sqrt{c}}-x_{0}\right)/\la\right)
\]
 This gives $\int|\na u|^{2}=\sqrt{c}\int|\na Q|^{2}.$ Hence 
\[
c=a+b\sqrt{c}\int|\na Q|^{2},
\]
which yields 
\[
\sqrt{c}=\frac{1}{2}\left(b\|\na Q\|_{2}^{2}+\sqrt{b^{2}\|\na Q\|_{2}^{4}+4a}\right).
\]
Hence, 
\begin{equation}
u(x)=\la^{-1/2}Q\left(\left(\frac{x}{b\|\na Q\|_{2}^{2}+\sqrt{b^{2}\|\na Q\|_{2}^{4}+4a}}-x_{0}\right)/\lambda\right)\label{eq: solution expression}
\end{equation}
for some $x_{0}\in\R^{3}$ and some $\la>0$. This proves that (\ref{eq: limiting Kirchhoff})
has a unique positive energy solution up to translations and scalings.\end{proof}

We point out that $c$ depends only on $a,b$ since $Q$ is explicitly
given. In other words, $c$ is independent of the choice of the positive
solution $u$. As a consequence, we conclude that all the positive
energy solutions to (\ref{eq: limiting Kirchhoff}) is given by 
\[
S=\left\{ \la^{-1/2}Q\left(\left(\frac{x}{b\|\na Q\|_{2}^{2}+\sqrt{b^{2}\|\na Q\|_{2}^{4}+4a}}-y\right)/\lambda\right):\la>0,y\in\R^{3}\right\} .
\]

\subsection{Proof of nondegeneracy}

With no loss of generality, we assume that $u(x)=u(|x|)$ is the unique
positive radial energy solution to Eq. (\ref{eq: limiting Kirchhoff})
with $\la=1$ in (\ref{eq: solution expression}). Still write $c=a+b\int|\na u|^{2}$.
Keep in mind that $c$ depends only on $a,b$. Define ${\cal A}:D\to D$
by
\[
{\cal A}\var=-c\De\var-5u^{4}\var.
\]
It is straightforward to derive from \cite[Chapter 5]{Ambrosetti-Malchiodi-Book}
that 
\begin{equation}
\operatorname{Ker}{\cal A}={\rm span}\left\{ u/2+x\cdot\na u,u_{x_{1}},u_{x_{2}},u_{x_{3}}\right\} .\label{eq: known nondegeneracy results}
\end{equation}
Note that $u/2+x\cdot\na u=u/2+ru^{\prime}(r)$ with $r=|x|$ is a
radial function in $D$. For simplicity, denote $D_{\rad}=\{v\in D:v(x)=v(|x|)\}$. 

To prove the nondenegeracy of $\L_{+}$, first we prove

\begin{proposition}\label{prop: radial nondengeracy}Let $\L_{+}\var=0$
and $\var\in D_{\rad}$. Then $\var=\la\left(u/2+x\cdot\na u\right)$
for some $\la\in\R$. \end{proposition}
\begin{proof}
Direct computation shows that $u/2+x\cdot\na u=u/2+ru^{\prime}(r)$
is indeed a radial solution to equation $\L_{+}\var=0$. We have to
prove that $u/2+x\cdot\na u$ is the unique radial solution to equation
$\L_{+}\var=0$ in $D_{\rad}$ up to a constant. 

Let $\var\in D_{\rad}$ satisfy $\L_{+}\var=0$. It is equivalent
to 
\[
{\cal A}\var=2b\left(\int\na u\cdot\na\var\right)\De u.
\]
Write $e_{0}=u/2+x\cdot\na u$ for simplicity. Since $D_{\rad}$ is
a Hilbert space, denote by $D_{0}$ the orthogonal complement of $\R e_{0}$
in $D_{\rad}$. Then $\var=\la e_{0}+\tilde{\var}$ for some $\la\in\R$
and $\tilde{\var}\in D_{0}$. By a direct computation, we find that
$\int\na u\cdot\na e_{0}=0$. This implies $u\in D_{0}$. Moreover,
note that (\ref{eq: known nondegeneracy results}) implies that ${\cal A}$
is invertible on $D_{0}$. It follows from ${\cal A}e_{0}=0$ and
$\int\na u\cdot\na e_{0}=0$ that $\tilde{\var}$ satisfies 
\[
{\cal A}\tilde{\var}=2b\left(\int\na u\cdot\na\tilde{\var}\right)\De u=-\frac{2b}{c}\left(\int\na u\cdot\na\tilde{\var}\right)u^{5}.
\]

Claim that $\int\na u\cdot\na\tilde{\var}=0$. To this end, first
note that ${\cal A}u=-4u^{5}$. Thus, the above identity implies 
\[
{\cal A}\left(\tilde{\var}-\frac{b}{2c}\left(\int\na u\cdot\na\tilde{\var}\right)u\right)=0.
\]
We know $u\in D_{0}$. Hence, (\ref{eq: known nondegeneracy results})
yields 
\[
\tilde{\var}-\frac{b}{2c}\left(\int\na u\cdot\na\tilde{\var}\right)u=0.
\]
 This yields
\[
\na\tilde{\var}=\frac{b}{2c}\left(\int\na u\cdot\na\tilde{\var}\right)\na u.
\]
It follows
\[
\int\na u\cdot\na\tilde{\var}=\frac{b}{2c}\left(\int\na u\cdot\na\tilde{\var}\right)\int|\na u|^{2}.
\]
Since $c>b\int|\na u|^{2}$, we have $b\int|\na u|^{2}/(2c)<1/2$.
We deduce from the above equation that $\int\na u\cdot\na\tilde{\var}=0$.
This proves the claim. 

Therefore, ${\cal A}\tilde{\var}=0$. Recall $\tilde{\var}\in D_{0}$.
Applying (\ref{eq: known nondegeneracy results}) gives $\tilde{\var}=0$.
Thus, we obtain $\var=\la e_{0}$. The proof is complete. 
\end{proof}
Now we can prove the nondegeneracy part of Theorem \ref{thm: uniqueness and nondegeneracy}. 

\begin{proof}[Proof of nondegeneracy] Since we have proved Proposition
\ref{prop: radial nondengeracy}, the rest proof for the nondegeneracy
of $\L_{+}$ is standard. We refer the readers to \cite[Chapter 5]{Ambrosetti-Malchiodi-Book}
for details. \end{proof}

\end{document}